\documentclass[a4paper,12pt]{amsart}
\usepackage{amssymb,amsfonts,amsmath,amsthm,amscd}
\usepackage[all,cmtip]{xy}
\usepackage{color}

\newtheorem{theorem}{Theorem}[section]
\newtheorem{lemma}[theorem]{Lemma}
\theoremstyle{definition}
\newtheorem{definition}[theorem]{Definition}

\newtheorem{remark}[theorem]{Remark}

\newtheorem{manualtheoreminner}{Theorem}
\newenvironment{manualtheorem}[1]{%
  \renewcommand\themanualtheoreminner{#1}%
  \manualtheoreminner
}{\endmanualtheoreminner}
\theoremstyle{remark}

\numberwithin{equation}{section}

\begin{document}
\title[Criteria for nilpotency of fusion systems]{Criteria for nilpotency of fusion systems}

	
%    author one information
\author{Jie Jian}
\address{School of Mathematics and Statistics,
		Hubei University,
		Wuhan, $430062$, P. R. China}
\curraddr{}
\email{jjian22@163.com}
\thanks{}
	
\author{Jun Liao}
\address{School of Mathematics and Statistics,
		Hubei University,
		Wuhan, $430062$, P. R. China}
\curraddr{}
\email{jliao@hubu.edu.cn}
\thanks{}
	
\author{Heguo Liu}
\address{Department of Mathematics, Hainan University, Haikou, $570228$, China
	}
\curraddr{}
\email{ghliu@hainanu.edu.cn}
\thanks{}

%    \subjclass is required.
\date{}
	
\dedicatory{}
	
%    Abstract is required.
\begin{abstract}
Let $p$ be an odd prime and let $\mathcal{F}$ be a fusion system over a finite $p$-group $P$. A fusion system $\mathcal{F}$ is said to be nilpotent if $\mathcal{F}=\mathcal{F}_{P}(P)$.
In this paper we provide new criteria for saturated fusion systems $\mathcal{F}$ to be nilpotent, which can be viewed as extension of the $p$-nilpotency theorem of Glauberman and Thompson for fusion systems attributed to Kessar and Linckelmann. 
%[5, Ch. 8, Theorem 3.1]
%\cite[Theorem A]{KL}
\end{abstract}
	
\maketitle

\section{Introduction}\label{section1}
Let $p$ be a prime. Let $P$ be a finite $p$-group. A fusion system $\mathcal{F}$ over $P$ is a category whose objects are the subgroups of $P$, with the set $\operatorname{Hom}_{\mathcal{F} }(Q,R)$ of morphisms consisting of monomorphisms from $Q$ into $R$, and such that some weak axioms are satisfied \cite[Definition I.2.1]{AKO}. A fusion system $\mathcal{F}$ is saturated if it satisfies two more axioms \cite[Definition I.2.2]{AKO} which hold in $\mathcal{F}_{S} (G)$ as a consequence of Sylow's theorem. 
Many classic results on fusion in a Sylow subgroup $P$ of a finite group $G$ can be interpreted as results about the fusion system $\mathcal{F}_{P} (G)$.
A theorem of Frobenius \cite[Theorem 8.6]{Glauberman71} says that $G$ is $p$-nilpotent if and only if $P$ controls $G$-fusion in $P$, that is to say, if and only if $\mathcal{F}_{P} (G) = \mathcal{F}_{P} (P )$.

Let  $J(P)$ be the Thompson subgroup of $P$ generated by all abelian subgroups of $P$ of maximal order. 
A finite group $G$ is called $p$-stable \cite[Page 22, Definition]{Glauberman71} if $gC_{G}(Q)$ lies in $O_{p}(N_{G}(Q)/C_{G}(Q))$
whenever $Q$ is a $p$-subgroup of $G$, $g\in N_{G}(Q)$ and $[Q,g,g] = 1$.
The well known Glauberman's ZJ theorem \cite[Theorem A]{Glauberman68} states that if $p$ is an odd prime, $P$ is a Sylow $p$-subgroup of a $p$-stable group $G$, and $C_{G}(O_{p}(G))\leq O_{p}(G)$, then $Z(J(P))$ is a characteristic subgroup of $G$.
Glauberman's ZJ theorem plays a major role in the classification of simple groups having abelian or dihedral Sylow 2-subgroups. There are several versions of the theorem, depending on how one defines the Thompson subgroup. 
M. K\i zmaz proves that in \cite[Theorem B]{Kizmaz} if $p$ is an odd prime, $P$ is a Sylow $p$-subgroup of a $p$-stable group $G$, $C_{G}(O_{p}(G))\leq O_{p}(G)$ and $D$ is a strongly closed subgroup in $P$, then $Z(J(D))$ is a normal subgroup of $G$. Recently, D. Allcock provides a new proof of Glauberman's ZJ theorem, in a form that clarifies the choices involved and offers more choices than classical treatments, see \cite[Theorem 1.1 and Corollary 1.5]{Allcock}.

In particular, the ZJ-type theorems say that ZJ-type subgroups control fusion in groups.
Consequently, we may obtain $p$-nilpotency criteria for groups.
The $p$-nilpotency theorem of Glauberman and Thompson \cite[Theorem 8.3.1]{Gorenstein} states that $G$ is $p$-nilpotent if and only if $N_{G}(Z(J(P)))$ is $p$-nilpotent, where $p$ is an odd prime. The $p$-nilpotency theorem has been generalised to $p$-blocks of finite groups \cite[Theorem]{KL03} and to arbitrary fusion systems \cite[Theorem A]{KL} by Kessar and Linckelmann. 
In this paper we will  give some new nilpotency criteria for fusion systems which is an extension of Glauberman-Thompson $p$-nilpotency theorem as well as of Kessar-Linckelmann nilpotency theorem for fusion systems \cite[Theorem A]{KL}.

Let $\mathfrak{A}b(P)$ be the set of all abelian subgroups of $P$ and $\mathcal{A}\subseteq \mathfrak{Ab}(P)$.  We set 
\[J_{\mathcal{A}}=\langle A\mid A\in \mathcal{A}\rangle,\quad I_{\mathcal{A}}=\bigcap_{A\in \mathcal{A}} A \mbox{    where $I_{\mathcal{A}}=1$ if $\mathcal{A}=\emptyset$}.\] 
$J_{\mathcal{A}}$ is a sort of generalised Thompson subgroup, and $I_{\mathcal{A}}$ lies in its center.
For every subgroup $Q$ of $P$, we define $\mathcal{A}|Q$ as $\{A\in \mathcal{A}\mid A\leq Q\}$.

Our aim is to give two new criteria for saturated fusion systems $\mathcal{F}$ to be nilpotent, which generalised the $p$-nilpotency theorem of Glauberman and Thompson for fusion systems of Kessar and Linckelmann. 

\begin{manualtheorem}{A}\label{main}
Let $\mathcal{F}$ be a saturated fusion system over a finite $p$-group $P$, where $p$ is an odd prime. Then $\mathcal{F}=\mathcal{F}_{P}(P)$ if and only if $N_{\mathcal{F}}(I_{\mathcal{A}})=\mathcal{F}_{P}(P)$, where $\mathcal{A}\subseteq \mathfrak{Ab}(P)$  satisfies the following properties:
\begin{itemize}
\item[(i)] For every $Q\unlhd P$, $I_{\mathcal{A}|Q}$ is $\operatorname{Aut}_{\mathcal{F}}(Q)$-invariant.
\item[(ii)] For every $B\unlhd P$ with nilpotent class at most two, if there exist members of $\mathcal{A}$ that contain $[B,B]$ but not $B$, then $B$ normalises one of them.
\end{itemize}
\end{manualtheorem}

\begin{manualtheorem}{B}\label{main2}
Let $\mathcal{F}$ be a saturated fusion system over a finite $p$-group $P$, where $p$ is an odd prime. Then $\mathcal{F}=\mathcal{F}_{P}(P)$ if and only if $N_{\mathcal{F}}(I_{\mathcal{A}|D})=\mathcal{F}_{P}(P)$ for some strongly closed subgroup $D$ in $\mathcal{F}$, $\mathcal{A}\subseteq \mathfrak{Ab}(P)$  satisfy the following properties:
\begin{itemize}
\item[(i)] For every $Q\unlhd P$, $I_{\mathcal{A}|Q}$ is $\operatorname{Aut}_{\mathcal{F}}(Q)$-invariant.
\item[(ii)] For every $B\unlhd P$ with nilpotent class at most two, if there exists $A\in \mathcal{A}$ such that $[B,B]\leq A$ and $B$ does not normalise $A$, then $A^{*}=(A\cap A^{b})[A,b]\in \mathcal{A}$ for every $b\in N_{B}(N_{P}(A))-N_{B}(A)$. %Then $A^{*}\in \mathcal{A}$. such that $A^{*}\cap B$ is maximal.
\end{itemize}
\end{manualtheorem}

\section{Background on fusion systems and finite groups}\label{section2}

In this section we collect some basic concepts and results that will be needed later. We refer to  \cite{AKO} for notations of fusion systems. 

Let $p$ be a prime and let $P$ be a finite $p$-group.
We first list some important classes of subgroups in a given fusion system.

\begin{definition}[{\cite[Definition I.2.2, 3.1, 4.1]{AKO}}]
Let $\mathcal{F}$ be a fusion system over $P$. Fix a subgroup $Q$ of $P$.
	
$\bullet$ $Q, R \leq P$ are $\mathcal{F}$-conjugate if they are isomorphic as objects of the category $\mathcal{F}$. Let $Q^{\mathcal{F}}$ denote the set of all subgroups of $P$ which are $\mathcal{F}$-conjugate to $Q$.
		
$\bullet$  $Q \leq P$ is $\mathcal{F}$-centric if $C_{P}(R)=Z(R)$ for all $R\in Q^{\mathcal{F}}$.
		
$\bullet$ $Q$ is normal in $\mathcal{F}$ (denoted $Q\lhd \mathcal{F}$) if for all $R, S\leq P$ and all $\varphi \in \operatorname{Hom}_{\mathcal{F}}(R, S)$, $\varphi$ extends to a morphism $\tilde{\varphi}\in \operatorname{Hom}_{\mathcal{F}}(RQ, SQ)$ such that $\tilde{\varphi}(Q)=Q$.
	
$\bullet$ $Q$ is strongly closed in $\mathcal{F}$ if no element of $Q$ is $\mathcal{F}$-conjugate to an element of $P\setminus Q$.
\end{definition}

Let  $\mathcal{F}$  be a saturated fusion system over $P$. Recall that $\mathcal{F}$ is constrained if there is a subgroup $Q\lhd \mathcal{F}$ which is $\mathcal{F}$-centric. If $\mathcal{F}$ is constrained, then a model for $\mathcal{F}$ is a finite group $G$ such that $P\in \operatorname{Syl}_{p}(G)$, $\mathcal{F}_{P}(G)=\mathcal{F}$ and $C_{G}(O_{p}(G))\leq O_{p}(G)$.  The model theorem \cite[Theorem I.4.9]{AKO} states that each constrained fusion system has models, and that they are unique up to isomorphism.

Note that $\mathcal{F}$ is nilpotent if $\mathcal{F}=\mathcal{F}_{P}(P)$. A non-nilpotent saturated fusion system $\mathcal{F}$ over $P$ is called sparse if the only proper fusion subsystem of $\mathcal{F}$ over $P$ is $\mathcal{F}_{P}(P)$. 

\begin{theorem}[{\cite[Theorem 3.5]{Glesser}}]\label{sparse}
Let $\mathcal{F}$ be a sparse fusion system over a finite $p$-group $P$.
If $p$ is odd, then $\mathcal{F}$ is constrained.
\end{theorem}

\begin{definition}{\cite[Definition 1.1]{Kizmaz}}
Let $G$ be a finite group, let $P$ be a Sylow $p$-subgroup of $G$, and let $D$ be a nonempty subset of $P$. We say that $D$ is a strongly closed subset in $P$ (with respect to $G$) if for all $U\subseteq D$ and $g\in G$ such that $U^{g} \subseteq P$, we have $U^{g} \subseteq D$.
\end{definition}

\begin{definition}{\cite[Page 22, Definition]{Glauberman71}}
A finite group $G$ is called $p$-stable if $gC_{G}(Q)$ lies in $O_{p}(N_{G}(Q)/C_{G}(Q))$ whenever $Q$ is a $p$-subgroup of $G$, $g\in N_{G}(Q)$ and $[Q,g,g] = 1$.
\end{definition}

\begin{lemma}[{\cite[9.4.5]{KS}}]\label{stable}
Let $p$ be an odd prime. A group $G$ is $p$-stable if $G$ has abelian Sylow $2$-subgroups.
\end{lemma}

\begin{lemma}[{\cite[3.2.8]{KS}}]\label{lemma3.2.8}
Let $N$ be a normal subgroup of $G$ with factor group $\bar{G}= G/N$, and let $P$ be a $p$-subgroup of $G$. Assume that $(|N|, p)=1$. Then $N_{\bar{G}}(\bar{P})=\overline{N_{G}(P)}$ and $C_{\bar{G}}(\bar{P})=\overline{C_{G}(P)}$.
\end{lemma}

\begin{theorem}[{\cite[Theorem 3.23]{Isaacs}}]\label{coprime-action}
Let $A$ act via automorphisms on $G$, where $A$ and $G$ are finite groups, and suppose that $(|A|, |G|)=1$. Assume also that at least one of $A$ or $G$ is solvable. Then for each prime $p$,  there exists an $A$-invariant Sylow $p$-subgroup of $G$. 
\end{theorem}

\begin{theorem}[{Frobenius, \cite[Theorem 5.26]{Isaacs}}]\label{Frobenius}
Let $G$ be a finite group, and suppose $p$ is a prime. Then the following are equivalent.
\begin{itemize}
\item[(i)] $G$ is $p$-nilpotent. 
	
\item[(ii)] $N_{G}(Q)$ is $p$-nilpotent for every nonidentity $p$-subgroup $Q\leq G$. 
	
\item[(iii)] $N_{G}(Q)/C_{G}(Q)$ is a $p$-group for every $p$-subgroup $Q\leq G$.
\end{itemize}	

\end{theorem}

\begin{theorem}[{Thompson, \cite[Theorem]{Thompson}}]\label{Thompson-p}
Let $G$ be a finite group and $p$ be an odd prime. Let $P$ be a Sylow $p$-subgroup of $G$. If $C_{G}(Z(P))$ and $N_{G}(J(P))$ are $p$-nilpotent, then $G$ is $p$-nilpotent.
\end{theorem}

D. Allcock generalizes the Thompson-Glauberman replacement theorem and gives an axiomatic version of the Glauberman’s ZJ-theorem in \cite{Allcock}.

\begin{theorem}[{\cite[Theorem 3.1]{Allcock}}]\label{replacement}
Let $p$ be a prime, $P$ be a finite $p$-group and $B\unlhd P$. If $p=2$ then assume $B$ is abelian. Suppose $A\leq P$ is abelian and contains $[B, B]$. Then either $B$ normalises $A$, or there exists $b\in N_{B}(N_{P}(A))-N_{B}(A)$.
For any such $b$, $A^{*}=(A\cap A^{b})[A,b]\leq AA^{b}$ has the properties
\begin{itemize}
\item[(i)] $A^{*}$ is abelian and contains $[B, B]$.	
\item[(ii)] $A^{*}\cap B$ strictly contains $A\cap B$ and is a proper subgroup of $B$.
\item[(iii)] $A^{*}$ and $A$ normalise each other.
\item[(iv)] $|A^{*}|=|A|$. 
\end{itemize}
\end{theorem}

\begin{remark}\label{remark2.10}
If we choose $A^{*}$ in Theorem \ref{replacement} such that $A^{*}\cap B$ is maximal, then $B$ normalises $A^{*}$.
\end{remark}

\begin{theorem}[{\cite[Theorem 1.1]{Allcock}}]\label{ZJ-axiomatic}
Suppose $p$ is a prime, $P$ is a finite $p$-group and $G$ is a finite group satisfying 
\begin{itemize}
\item[(a)] $P$ is a Sylow $p$-subgroup of $G$.
\item[(b)] $C_{G}(O_{p}(G))\leq O_{p}(G)$.
\item[(c)] $G$ acts $p$-stably on every normal $p$-subgroup of $G$.
\end{itemize}
Then $I_{\mathcal{A}}\unlhd G$ if $\mathcal{A}\subseteq \mathfrak{Ab}(P)$ satisfies the following properties:
\begin{itemize}
\item[(i)] For every $Q\unlhd P$, $I_{\mathcal{A}|Q}$ is $N_{G}(Q)$-invariant.
\item[(ii)] For every $B\unlhd P$ with nilpotent class at most two, if there exist members of $\mathcal{A}$ that contain $[B,B]$ but not $B$, then $B$ normalises one of them.
\end{itemize}
Furthermore, $I_{\mathcal{A}}$ is characteristic in $G$ if it  is characteristic in $P$.
\end{theorem}

\section{Preliminary Lemmas}\label{section3}

Before we prove the main theorems, we require some necessary lemmas.
\begin{lemma}\label{p-complement}
Let $p$ be an odd prime and $P$ be a Sylow $p$-subgroup of a finite group $G$. Suppose that $\mathcal{A}\subseteq \mathfrak{Ab}(P)$ satisfies the following properties:
\begin{itemize}
\item[(i)]  For every $Q\unlhd P$, $I_{\mathcal{A}|Q}$ is $N_{G}(Q)$-invariant.
\item[(ii)] For every $B\unlhd P$ with nilpotent class at most two, if there exist members of $\mathcal{A}$ that contain $[B,B]$ but not $B$, then $B$ normalises one of them.
\end{itemize}
If $N_{G}(I_{\mathcal{A}})$ is $p$-nilpotent, then $G$ is $p$-nilpotent.
\end{lemma}

\begin{proof}
Suppose that $G$ is a minimal counterexample with respect to order of group. Then $G$ is not $p$-nilpotent.
Let $\mathcal{W}$ be the set of all nonidentity $p$-subgroup $W$ of $G$ such that $N_{G}(W)$ is not $p$-nilpotent.
By Theorem \ref{Frobenius}, there exists  $W\in \mathcal{W}$ as $G$ is not $p$-nilpotent.

\textbf{Step 1:}  If $P\leq H<G$, then $H$ is $p$-nilpotent.

Note that 	$N_{H}(Q)\leq N_{G}(Q)$ for $Q\leq P$. 
Hence the conditions of the lemma hold and $N_{H}(I_{\mathcal{A}})$ is $p$-nilpotent.
By the minimality of assumption, $H$ is $p$-nilpotent.
	
\textbf{Step 2:}  $O_{p'}(G)=1$.

Set $\bar{G}=G/O_{p'}(G)$. Then $\bar{P}$ is a Sylow $p$-subgroup of $\bar{G}$ and $I_{\bar{\mathcal{A}}}=\overline{I_{\mathcal{A}}}$.
By Lemma \ref{lemma3.2.8}, $N_{\bar{G}}(\overline{I_{\mathcal{A}}})=\overline{N_{G}(I_{\mathcal{A}})}$.
Thus, $N_{\bar{G}}(I_{\bar{\mathcal{A}}})$ is $p$-nilpotent.
For $\bar{Q}\unlhd \bar{P}$, we have $Q\unlhd P$ and $N_{\bar{G}}(\bar{Q})=\overline{N_{G}(Q)}$.
Hence $\bar{G}$ satisfies the conditions of the lemma.
If  $\bar{G}$ is $p$-nilpotent, then $\bar{G}$ is a $p$-group. Hence $G=PO_{p'}(G)$ and $G$ is $p$-nilpotent, a contradiction.
Thus $\bar{G}$ is not $p$-nilpotent. So $|\bar{G}|=|G|$ by minimality of $G$.
Hence  $O_{p'}(G)=1$.
	
\textbf{Step 3:} $O_{p}(G)\neq 1$.

Let $\mathcal{U}$ be the set of all subgroup $W\in \mathcal{W}$ such that $|N_{G}(W)|_{p}$ is maximal. 
We choose $U\in \mathcal{U}$ such that $|U|$ is maximal. 	
Let $N=N_{G}(U)$, and $S\in \operatorname{Syl}_{p}(N)$.
There is $g\in G$ such that $S^{g}\leq P$.
Replace $U$ by $U^{g}$, we may assume that $S\leq P$.  
Since $N$ is not $p$-nilpotent, it follows that  either $C_{N}(Z(S))$ or $N_{N}(J(S))$ is not $p$-nilpotent by Theorem \ref{Thompson-p}.
So there is $X\in \{Z(S), J(S)\}$ such that $N_{G}(X)$ is not $p$-nilpotent as a subgroup of $p$-nilpotent groups is again $p$-nilpotent.
Therefore $X\in \mathcal{W}$. 
Assume that $N<G$. By Step 1, $S$ is properly contained in $P$ as $N$ is not $p$-nilpotent.
Let $T=N_{P}(S)$. Then $S\lhd T$ and $S<T$.
Since $X$ is characteristic in $S$ and $S\lhd T$, we have $X\lhd T$. So $T\leq  N_{G}(X)$. 
Thus, we have 
$|N_{G}(U)|_{p}=|S|<|T|\leq |N_{G}(X)|_{p}$.
This contradicts the choice of $U$, we conclude that $N=G$.
So $U\unlhd G$.
It is easy to see that $O_{p}(G)\in \mathcal{U}$ and $U\leq O_{p}(G)$.
By the choice of $U$, we have $U=O_{p}(G)$. In particular, $O_{p}(G)\neq 1$.
	
\textbf{Step 4:}  $G/O_{p}(G)$ is $p$-nilpotent and $C_{G}(O_{p}(G))\leq O_{p}(G)$.
	
Let $\bar{G}=G/O_{p}(G)$. By Theorem \ref{Thompson-p}, it suffices to show that both $C_{\bar{G}}(Z(\bar{P}))$ and $N_{\bar{G}}(J(\bar{P}))$ are $p$-nilpotent. 
If $P=O_{p}(G)$, then $\bar{G}$ is $p$-nilpotent.
Now we assume that $P\neq O_{p}(G)$. Then $Z(\bar{P})$ and $J(\bar{P})$ are nontrivial. 
Let $X \leq P$ such that
$\bar{X}$ is either $Z(\bar{P})$ or $J(\bar{P})$.
Since  $\bar{X}\unlhd \bar{P}$, we see that $X\unlhd P$ and $P\leq N_{G}(X)$.
Thus, we have $|N_{G}(X)|_{p}=|P|=|G|_{p}$.
By Step 3, $O_{p}(G)$ is maximal in $\mathcal{U}$. Note that $O_{p}(G)<X$ implies $X\notin \mathcal{W}$.
Thus $N_{G}(X)$ is $p$-nilpotent. 
Since $X$ contains $O_{p}(G)$, we get that  $\overline{N_{G}(X)}=N_{\bar{G}}(\bar{X})$. So $N_{\bar{G}}(\bar{X})$ is $p$-nilpotent. 
Therefore, both $C_{\bar{G}}(Z(\bar{P}))$ and $N_{\bar{G}}(J(\bar{P}))$ are $p$-nilpotent. 
By Theorem \ref{Thompson-p}, $\bar{G}$ is $p$-nilpotent.

Since $G/O_{p}(G)$ is $p$-nilpotent, it follows that $G$ is $p$-solvable. Note that $O_{p'}(G)=1$ by Step 2. Thus by \cite[Theorem 6.3.2]{Gorenstein}, $C_{G}(O_{p}(G))\leq O_{p}(G)$.
	
\textbf{Step 5:} $G$ is not $p$-stable.
	
Suppose that  $G$ is a $p$-stable group, then $I_{\mathcal{A}}\unlhd G$ by Theorem \ref{ZJ-axiomatic}. So $G=N_{G}(I_{\mathcal{A}})$ is $p$-nilpotent, a contradiction. 
	
\textbf{Step 6:} Final contradiction.
	
If $G$ is not a $p$-stable group, then $G$ has non-abelian Sylow $2$-subgroups by Theorem \ref{stable}.
By Step 4, $G/O_{p}(G)$ is $p$-nilpotent. So that $G/O_{p}(G)$ has a normal $p$-complement $\bar{K}$ for some $K\unlhd G$.
Then $\bar{K}$ has non-abelian Sylow $2$-subgroups. 
Since the action of $\bar{P}$ on $\bar{K}$  is a coprime action, it follows that there exists a $\bar{P}$-invariant Sylow $2$-subgroup $\bar{T}$ of $\bar{K}$ by Theorem \ref{coprime-action}. 
Then  $Z(\bar{T})$ is also $\bar{P}$-invariant.
Let $V$ be the inverse image of $\bar{P}Z(\bar{T})$ in $G$. 
Then $P< V<G$ as $\bar{T}$ is non-abelian. Thus $V$ is $p$-nilpotent by Step 1. Hence $V$ has a normal $p$-complement $V_{0}$.
Since $[V_{0}, O_{p}(G)]\leq  V_{0}\cap O_{p}(G)=1$,
we get that $V_{0} \leq C_{G}(O_{p}(G))$. By Step 4, $C_{G}(O_{p}(G))\leq O_{p}(G)$.
Thus $V_{0} \leq O_{p}(G)$ and $V_{0}=1$,  a contradiction.
This final contradiction shows that $G$ is not a counterexample and $G$ is $p$-nilpotent.
\end{proof}

We generalise this lemma to the case of strongly closed subgroups.

\begin{lemma}\label{strongly-closed1}
Let $p$ be an odd prime and let $P$ be a Sylow $p$-subgroup of a $p$-stable group $G$. Suppose that $\mathcal{A}\subseteq \mathfrak{Ab}(P)$ satisfies the following properties:
\begin{itemize}
\item[(i)]  For every $Q\unlhd P$, $I_{\mathcal{A}|Q}$ is $N_{G}(Q)$-invariant.
\item[(ii)] For every $B\unlhd P$ with nilpotent class at most two, if there exists $A\in \mathcal{A}$ that $[B,B]\leq A$ and $B$ does not normalise $A$, then $A^{*}=(A\cap A^{b})[A,b]\in \mathcal{A}$ for every $b\in N_{B}(N_{P}(A))-N_{B}(A)$. 
%such that $A^{*}\cap B$ is maximal.
\end{itemize}	
If $D$ is a strongly closed subgroup in $P$,  and $B\unlhd G$ is a normal $p$-subgroup of $G$, then $I_{\mathcal{A}|D}\cap B\unlhd G$.
\end{lemma}
\begin{proof}
If $B=1$, then it clearly holds.
Assume that $B\neq 1$. Let $G$ be a counterexample, and choose $B$ to be the smallest possible normal $p$-subgroup contradicting with the lemma. 
Let $N$ be the largest normal subgroup of $G$ that normalises $I_{D}\cap B$. Let $D^{*}=D\cap N$, $I_{D}=I_{\mathcal{A}|D}$  and $I_{D^{*}}=I_{\mathcal{A}|D^{*}}$. 
	
\textbf{Step 1:}  $B=(I_{D}\cap B)^{G}$.

Set $B_{1}=(I_{D}\cap B)^{G}$, which is generated by all the conjugates of $I_{D}\cap B$ in $G$. Clearly,  $B_{1}\leq B$. If $B_{1}<B$, then  $I_{D}\cap B_{1}\unlhd G$ by the choice of $B$.
Since $I_{D}\cap B\leq B_{1}$, we have $I_{D}\cap B\leq I_{D}\cap B_{1}\leq I_{D}\cap B$. Hence $I_{D}\cap B=I_{D}\cap B_{1}$, a contradiction.
Thus, $B_{1}=B$.
	
\textbf{Step 2:} $[B,B]\leq I_{D}$ and $B$ is nilpotent of class at most two. 
	
Clearly $B'=[B,B]<B$. By the choice of $B$, $I_{D}\cap B'\unlhd G$. Since $D$ is a strongly closed subgroup in $P$, we have $D^{x}=D$ if $x\in P$.
So $D\unlhd P$, and thus $I_{D}$ is $N_{G}(D)$-invariant by condition (i). In particular, $I_{D}$ is $P$-invariant.	
Hence $I_{D}$ and $B$ normalise each other. So we have $[I_{D}\cap B,B]\leq I_{D}\cap B'$. Since $B$ and $ I_{D}\cap B'$ are both normal subgroups of $G$, we obtain $[(I_{D}\cap B)^{g},B]\leq I_{D}\cap B'$ for all $g\in G$. This yields $B'=[B,B]=[(I_{D}\cap B)^{G},B]\leq I_{D}\cap B'$. Hence $B'\leq I_{D}$, and $[I_{D}\cap B,B']=1$.  Since $B=(I_{D}\cap B)^{G}$, it follows that $[B,B']=1$. So we see that $B$ is nilpotent of class at most two. 
	
\textbf{Step 3:}  $G=N_{G}(D^{*})N=N_{G}(I_{D^{*}})N$. 
	
Let $Q=P\cap N$. Then $Q\in \operatorname{Syl}_{p}(N)$. 
By Frattini argument, we have $G=N_{G}(Q)N$. 
Let $U\subseteq D^{*}$ and $g\in G$ such that $U^{g}\subseteq P$. It follows that $U^{g}\subseteq D$ as $U\subseteq D$ and $D$ is strongly closed in $P$. Since $N\unlhd G$ and $U\subseteq D^{*}\subseteq N$, we see that $U^{g}\subseteq N$. Therefore, $U^{g}\subseteq D\cap N=D^{*}$. Then $D^{*}$ is also strongly closed in $P$.	
If $x\in N_{G}(Q)$, then $(D^{*})^{x}\leq Q\leq P$. Since $D^{*}$ is strongly closed in $P$, we have $(D^{*})^{x}\leq D^{*}$. Hence $x\in N_{G}(D^{*})$. Thus $G=N_{G}(D^{*})N$. Since $D^{*}$ is strongly closed in $P$, we see that $D^{*}\unlhd P$. Thus, $I_{D^{*}}$ is $N_{G}(D^{*})$-invariant by condition (i). So $N_{G}(D^{*})\leq N_{G}(I_{D^{*}})$. Hence, $G=N_{G}(I_{D^{*}})N$.
	
\textbf{Step 4:} There exists an element of $\mathcal{A}|D$ that is not contained in $D^{*}$.
	
Assume that all the elements of the set $\mathcal{A}|D$ are contained in $D^{*}$. Then $\mathcal{A}|D$ is equal to $\mathcal{A}|D^{*}$, and $I_{D}=I_{D^{*}}$. Recall that  $D^{*}\unlhd P$, then $I_{D^{*}}$ is $N_{G}(D^{*})$-invariant by condition (i). Thus, $I_{D}$ is $N_{G}(D^{*})$-invariant. Hence $I_{D}\cap B$ is $N_{G}(D^{*})$-invariant. By Step 3, we have $G=N_{G}(D^{*})N$. 
Note that $N\leq N_{G}(I_{D}\cap B)$. So $I_{D}\cap B\unlhd G$, a contradiction.
	
\textbf{Step 5:} If $A\in \mathcal{A}|D$ and $B$ normalises $A$, then $A\leq D^{*}$

Since $[B,A]\leq A$ and $A$ is abelian, it follows that $[B,A,A]=1$.  
Since $G$ is $p$-stable and $ B\unlhd G$, we have that $AC/C \leq O_{p}(G/C)$ where $C=C_{G}(B)$. 
Note that $C$ normalises $I_{D}\cap B$ and $C\unlhd G$, and so $C\leq N$ by the choice of $N$. 
Let $L\unlhd G$ such that $L/C=O_{p}(G/C)$. Then $L\unlhd G$, $L/C$ is a $p$-group and $A\leq L$.
Then $C\leq N$ implies that $LN/N$ is a normal $p$-subgroup of $G/N$.
It follows that $AN/N \leq LN/N\leq O_{p}(G/N)$. 
Let $M\unlhd G$ such that $M/N = O_{p}(G/N)$. 
Then $M/N\leq PN/N$, and $M=M\cap PN=(M\cap P)N$. 
Hence $M$ normalises $I_{D}\cap B$ as $I_{D}$ is $P$-invariant. So $M\leq N$ by the choice of $N$. Thus, $O_{p}(G/N)=1$ and $A\leq N$.
Therefore, we have $A\leq D^{*}$ as $A\leq D$.

\textbf{Step 6:} Final contradiction.

By Step 4, there exists an element of the set $\mathcal{A}|D$ that is not contained in $D^{*}$. Choose $A\in \mathcal{A}|D$ such that $A \nleq D^{*}$. Then Step 5,  $B$ does not normalise $A$.
By Step 2, $[B,B]\leq I_{D}$ and so $[B,B]\leq A$.
By Theorem \ref{replacement}, there exists  $b\in N_{B}(N_{P}(A))-N_{B}(A)$ such that $A^{*}\cap B>A\cap B$, where $A^{*}=(A\cap A^{b})[A,b]$.
By condition (ii), $A^{*}\in \mathcal{A}$.
Since $D$ is strongly closed in $P$ and $A\leq D$, we have $A^{*}\leq AA^{b}\leq D$. Thus, $A^{*}\in \mathcal{A}|D$. 
If $B$ does not normalise $A^{*}$, then replace $A$ by $A^{*}$ and continue to apply Theorem \ref{replacement}. By induction, we may assume that  $B$ does not normalise $A$, and $B$ normalises $A^{*}$,
where $A, A^{*}=(A\cap A^{b})[A,b]\in  \mathcal{A}|D$ such that $A^{*}$ and $A$ normalise each other.
Since  $A^{*}\in \mathcal{A}|D$ and $B$ normalises $A^{*}$, by Step 5 we have $A^{*}\leq D^{*}$. So $I_{D^{*}}\leq A^{*}\leq N_{P}(A)$.
Note that $I_{D}\cap B\leq I_{D}\leq I_{D^{*}}$. By Step 3, $G=N_{G}(I_{D^{*}})N$. We have $B=(I_{D}\cap B)^{G}=(I_{D}\cap B)^{N_{G}(I_{D^{*}})}\leq (I_{D^{*}})^{N_{G}(I_{D^{*}})}=I_{D^{*}}$.
	Thus, $B\leq I_{D^{*}}\leq A^{*}\leq N_{P}(A)$, a contradiction. 
This final contradiction shows that $G$ is not a counterexample, and $I_{\mathcal{A}|D}\cap B\unlhd G$.
\end{proof}

\begin{lemma}\label{strongly-closed2}
Let $p$ be an odd prime and let $P$ be a Sylow $p$-subgroup of a $p$-stable group $G$. Suppose that $C_{G}(O_{p}(G))\leq O_{p}(G)$ and $\mathcal{A}\subseteq \mathfrak{Ab}(P)$ satisfies the following properties:
\begin{itemize}
\item[(i)]  For every $Q\unlhd P$, $I_{\mathcal{A}|Q}$ is $N_{G}(Q)$-invariant.
\item[(ii)] For every $B\unlhd P$ with nilpotent class at most two, if there exists $A\in \mathcal{A}$ that $[B,B]\leq A$ and $B$ does not normalise $A$, then $A^{*}=(A\cap A^{b})[A,b]\in \mathcal{A}$ for every $b\in N_{B}(N_{P}(A))-N_{B}(A)$. %such that $A^{*}\cap B$ is maximal.
\end{itemize}
If $D$ is a strongly closed subgroup in $P$, then $I_{\mathcal{A}|D}\unlhd G$.
\end{lemma}

\begin{proof}
Let $I_{D}=I_{\mathcal{A}|D}$. Since $D$ is a strongly closed subgroup in $P$, it follows that  $D\unlhd P$. So  $I_{D}$ is $N_{G}(D)$-invariant by condition (i). In particular, $I_{D}$ is $P$-invariant. Thus, $I_{D}\unlhd P$. Since $O_{p}(G)\leq P$ normalises $I_{D}$ and $I_{D}$ is abelian, we see that $[O_{p}(G),I_{D},I_{D}]=1$. Let $C=C_{G}(O_{p}(G))$. Since $G$ is $p$-stable, it follows that $I_{D}C/C\leq O_{p}(G/C)$. 
Note that we have $C=C_{G}(O_{p}(G))\leq O_{p}(G)$ by hypothesis, so $O_{p}(G/C)=O_{p}(G)/C$.
Thus $I_{D}C/C\leq O_{p}(G)/C$ and  $I_{D}\leq O_{p}(G)$.
By Lemma \ref{strongly-closed1}, we have $I_{D}=I_{D}\cap O_{p}(G)\unlhd G$.
\end{proof}

Similar to the proof of Lemma \ref{p-complement}, we  obtain the following result.

\begin{lemma}\label{lem4}
Let $p$ be an odd prime and let $P$ be a Sylow $p$-subgroup of a finite group $G$. Suppose that $\mathcal{A}\subseteq \mathfrak{Ab}(P)$  satisfies the following properties:
\begin{itemize}
\item[(i)]  For every $Q\unlhd P$, $I_{\mathcal{A}|Q}$ is $N_{G}(Q)$-invariant.
\item[(ii)] For every $B\unlhd P$ with nilpotent class at most two, if there exists $A\in \mathcal{A}$ that $[B,B]\leq A$ and $B$ does not normalise $A$, then $A^{*}=(A\cap A^{b})[A,b]\in \mathcal{A}$ for every $b\in N_{B}(N_{P}(A))-N_{B}(A)$.
\end{itemize}	
 Assume that $D$ is a strongly closed subgroup in $P$. If $N_{G}(I_{\mathcal{A}|D})$ is $p$-nilpotent, then $G$ is $p$-nilpotent.
\end{lemma}

\section{Proof of Theorem \ref{main} and Theorem \ref{main2}} \label{section4}
\begin{proof}[Proof of Theorem \ref{main}]
The necessity is clear.
So we just need to establish the sufficiency of the condition. Let $\mathcal{F}$ be a minimal counterexample such that the number of morphisms $|\mathcal{F}|$ of $\mathcal{F}$ is minimal and $N_{\mathcal{F}}(I_{\mathcal{A}})=\mathcal{F}_{P}(P)$ but $\mathcal{F}\neq \mathcal{F}_{P}(P)$. 
	
We claim that every proper fusion subsystem of $\mathcal{F}$ over $P$ is equal to $\mathcal{F}_{P}(P)$.
Let $\mathcal{E}$ be a proper fusion subsystem of $\mathcal{F}$ over $P$. 
Note that $I_{\mathcal{A}}$ is $\operatorname{Aut}_{\mathcal{F}}(P)$ invariant by property (i), so $I_{\mathcal{A}}\unlhd P$.
Then $\mathcal{F}_{P}(P)\subseteq N_{\mathcal{E}}(I_{\mathcal{A}}) \subseteq N_{\mathcal{F}}(I_{\mathcal{A}})=\mathcal{F}_{P}(P)$. Thus, $N_{\mathcal{E}}(I_{\mathcal{A}})=\mathcal{F}_{P}(P)$. Since $\operatorname{Aut}_{\mathcal{E}}(Q)\leq \operatorname{Aut}_{\mathcal{F}}(Q)$, it follows that $I_{\mathcal{A}|Q}$ is $\operatorname{Aut}_{\mathcal{E}}(Q)$ invariant if $Q\unlhd P$ by property (i).
By the minimality of $\mathcal{F}$, we have $\mathcal{E}=\mathcal{F}_{P}(P)$ as claimed.

Therefore, $\mathcal{F}$ is sparse. By Theorem \ref{sparse}, $\mathcal{F}$ is constrained as $p$ is odd.	
Hence there exists a finite group $G$ such that $P\in \operatorname{Syl}_{p}(G)$, $\mathcal{F}=\mathcal{F}_{P}(G)$ and $C_{G}(O_{p}(G))\leq O_{p}(G)$. Since $N_{\mathcal{F}}(I_{\mathcal{A}})=\mathcal{F}_{P}(N_{G}(I_{\mathcal{A}}))=\mathcal{F}_{P}(P)$, we have $N_{G}(I_{\mathcal{A}})$ is $p$-nilpotent. By Lemma \ref{p-complement}, $G$ is $p$-nilpotent. Hence $\mathcal{F}=\mathcal{F}_{P}(G)=\mathcal{F}_{P}(P)$, a contradiction. This shows that $\mathcal{F}=\mathcal{F}_{P}(P)$, and the proof is complete.
\end{proof}

\begin{proof}[Proof of Theorem \ref{main2}]
Suppose that $D$ is a strongly closed subgroup in $\mathcal{F}$.
Assume that $N_{\mathcal{F}}(I_{\mathcal{A}|D})=\mathcal{F}_{P}(P)$, and $\mathcal{A}\subseteq \mathfrak{Ab}(P)$ satisfies the properties (i) and (ii) in Theorem \ref{main2}.
Let  $\mathcal{F}$ be a minimal counterexample such that $\mathcal{F}\neq \mathcal{F}_{P}(P)$. 

We claim that $\mathcal{F}$ is sparse. In fact, let $\mathcal{E}$ be a proper fusion subsystem of $\mathcal{F}$ over $P$. We will prove that $\mathcal{E}=\mathcal{F}_{P}(P)$.
Recall that $D$ is strongly closed in $\mathcal{F}$, and so $D\unlhd P$. By property (i), $I_{\mathcal{A}|D}$ is $\operatorname{Aut}_{\mathcal{F}}(D)$-invariant. In particular, $I_{\mathcal{A}|D}$ is $P$-invariant and $I_{\mathcal{A}|D}\unlhd P$. Thus $N_{\mathcal{F}}(I_{\mathcal{A}|D})$ is a saturated fusion system over $P$.
Note that $\mathcal{F}_{P}(P)\subseteq N_{\mathcal{E}}(I_{\mathcal{A}|D}) \subseteq N_{\mathcal{F}}(I_{\mathcal{A}|D})=\mathcal{F}_{P}(P)$. Hence $N_{\mathcal{E}}(I_{\mathcal{A}|D})=\mathcal{F}_{P}(P)$.
By property (i), for each $Q\unlhd P$, $I_{\mathcal{A}|Q}$  is $\operatorname{Aut}_{\mathcal{F}}(Q)$ invariant. Since $\operatorname{Aut}_{\mathcal{E}}(Q)\leq \operatorname{Aut}_{\mathcal{F}}(Q)$, it follows that $I_{\mathcal{A}|Q}$ is $\operatorname{Aut}_{\mathcal{E}}(Q)$ invariant.
By the minimality of $\mathcal{F}$, $\mathcal{E}=\mathcal{F}_{P}(P)$. So $\mathcal{F}$ is sparse as claimed.

Since $p$ is an odd prime, it follows that the sparse fusion system $\mathcal{F}$ is constrained by Theorem \ref{sparse}. By the model theorem there exists a finite group $G$ such that $P\in \operatorname{Syl}_{p}(G)$, $\mathcal{F}=\mathcal{F}_{P}(G)$ and $C_{G}(O_{p}(G))\leq O_{p}(G)$. Since $N_{\mathcal{F}}(I_{\mathcal{A}|D})=\mathcal{F}_{P}(N_{G}(I_{\mathcal{A}|D}))=\mathcal{F}_{P}(P)$, we have $N_{G}(I_{\mathcal{A}|D})$ is $p$-nilpotent. By Lemma \ref{lem4}, $G$ is $p$-nilpotent. Hence $\mathcal{F}=\mathcal{F}_{P}(G)=\mathcal{F}_{P}(P)$, a contradiction. This contradiction shows that $\mathcal{F}=\mathcal{F}_{P}(P)$, and the proof is complete.
\end{proof}

\end{document}